\documentclass [3p, 12pt, sort&compress]{elsarticle}
\usepackage[numbers]{natbib}
\usepackage{graphicx}
\usepackage[small]{subfigure,epsfig}
\usepackage{indentfirst}
\usepackage{amsmath,latexsym,enumerate}
\usepackage{amssymb}
\usepackage{amsfonts}

\usepackage{amsthm}

\theoremstyle{plain} \newtheorem{theorem}{Theorem}
\newtheorem{lemma}{Lemma}
\numberwithin{equation}{section} \numberwithin{lemma}{section} \numberwithin{theorem}{section}
\makeatletter
\def\ps@pprintTitle{%
 \let\@oddhead\@empty
 \let\@evenhead\@empty
 \def\@oddfoot{\centerline{\thepage}}%
 \let\@evenfoot\@oddfoot}
\makeatother

\begin{document}
\begin{frontmatter}

\title{The method of Puiseux series and invariant algebraic curves}

\author{Maria V. Demina}

\address{National Research University Higher School of Economics, 34 Tallinskaya Street, 123458, Moscow, Russian Federation, maria\underline{ }dem@mail.ru}

%\ead{maria\underline{ }dem@mail.ru}

\begin{abstract}

An explicit expression for the cofactor related to an irreducible invariant
algebraic curve of a polynomial dynamical system in the plane is derived. A
sufficient condition for a polynomial dynamical system in the plane to have a
finite number of irreducible invariant algebraic curves is obtained. All
these results are applied to Li\'{e}nard dynamical systems $x_t=y$,
$y_t=-f(x)y-g(x)$ with $\deg f<\deg g<2\deg f+1$. The general structure of
their irreducible invariant algebraic curves and cofactors   is found. It is
shown that Li\'{e}nard dynamical systems  with $\deg f<\deg g<2\deg f+1$ can
have  at most two distinct irreducible invariant algebraic curves
simultaneously and, consequently, are not integrable with a rational first
integral.
\end{abstract}

\begin{keyword}
invariant algebraic curves, Darboux polynomials,  Li\'{e}nard systems

\end{keyword}

\end{frontmatter}

\section{Introduction}\label{Introduction}

The problem of finding invariant algebraic curves of polynomial dynamical
systems in the plane has twofold interest. Firstly, invariant algebraic
curves produce in explicit form  certain phase curves of related dynamical
systems. In particular, algebraic limit cycles are given by ovals of
invariant algebraic curves. By an oval of an algebraic curve we mean a
connected component lying in the projective plane $\mathbb{R}P^2$ and
diffeomorphic to a circle. Secondly, the complete set of irreducible
invariant algebraic curves is a key object in the study of Darboux and
Liouvillian integrability of dynamical systems \cite{Singer, Christopher02,
Christopher, Ilyashenko, Zhang, Lei01}.

Recently a novel method of finding  invariant algebraic curves, which we
shall call the method of Puiseux series, was introduced \cite{Demina07,
Demina06, Demina12}. This method was used to establish the general structure
of an irreducible invariant algebraic curve for any polynomial dynamical
system in the plane \cite{Demina07}. In this article we present a further
development of the method. We obtain some results concerning uniqueness of
invariant algebraic curves. We derive a sufficient condition for a polynomial
dynamical systems in the plane to have a finite number of irreducible
invariant algebraic curves. Another novel results  is an explicit expression
for the cofactor related to an irreducible invariant algebraic curve. This
expression is of importance in the theory of integrability. Indeed,
investigating the linear dependence or independence of the cofactors related
to irreducible invariant algebraic curves and exponential invariants enables
one to solve the problem of Darboux integrability for a given dynamical
system \cite{Singer, Christopher02, Christopher, Ilyashenko, Zhang}. In
addition, if the polynomial equal to the divergence of the vector field
related to the dynamical system is added to the set of cofactors, then the
Liouvillian integrability or non--integrability can be established
\cite{Singer, Christopher02, Christopher, Ilyashenko, Zhang}.

%In addition, we discuss the way one can find the polynomial at

Further, we apply the method of Puiseux series to famous  Li\'{e}nard
dynamical systems $x_t=y$, $y_t=-f(x)y-g(x)$ having a great number of
physical applications.
 The integrability and limit cycles of  Li\'{e}nard dynamical systems
have been intensively studied in recent years \cite{Odani, Zoladec01,
Llibre05, Llibre07, Zhang02, Llibre04, Liu01, Llibre06, Cheze01, Ignatyev01,
Demina06, Demina07, Demina11}. K. Odani was the first to show that
Li\'{e}nard dynamical systems with $\deg g\leq\deg f$ do not have invariant
algebraic curves excluding some trivial cases \cite{Odani}. Liouvillian
integrability of Li\'{e}nard dynamical systems with $\deg g\leq\deg f$ was
considered by J. Llibre and C. Valls \cite{Llibre07, Llibre06}, see also
\cite{Cheze01}. The study of hyperelliptic limit cycles, i.e. limit cycles
given by invariant algebraic curves of the form $(y+u(x))^2-v(x)=0$ with
$u(x)$ and $v(x)\in\mathbb{R}[x]$, in Li\'{e}nard dynamical systems was begun
by H. \.{Z}o{\l}\c{a}dec \cite{Zoladec01} and was further continued by
J.~Llibre and X. Zhang \cite{Llibre05}, X. Yu and X. Zhang \cite{Zhang02}, C.
Liu, G. Chen, and J.~Yang~\cite{Liu01}.

In this article we consider  Li\'{e}nard dynamical systems with $\deg f<\deg
g<2\deg f+1$. In particular, we give the general structure of their
irreducible invariant  algebraic curves and corresponding cofactors. We prove
that Li\'{e}nard dynamical systems  with $\deg f<\deg g<2\deg f+1$ can have
at most two distinct irreducible invariant algebraic curves simultaneously.
This fact means that Li\'{e}nard dynamical systems in question are not
integrable with a rational first integral. As an example, all irreducible
invariant algebraic curves for a family of  Li\'{e}nard differential system
with $\deg f=2$ and $\deg g=4$ are obtained.
 It turns out that there exist irreducible invariant algebraic curves of more complicated structure than it was supposed before. All these results seem to be new.

This article is organized as follows. In Section \ref{S:Puiseux} we give a
brief review of the method of Puiseux series and present novel results
including the explicit expression for the cofactor of an irreducible
invariant algebraic curve. In Section \ref{S:Lienard} we study Li\'{e}nard
dynamical systems  with $\deg f<\deg g<2\deg f+1$. In Section
\ref{S:Lienard_2_4} we present the classification of irreducible invariant
algebraic curves for a family of  Li\'{e}nard differential system with $\deg
f=2$ and $\deg g=4$.

\section{The method of Puiseux series}\label{S:Puiseux}

Let us consider the following  dynamical system in the plane
\begin{equation}
\begin{gathered}
 \label{DS}
 x_t=P(x,y),\quad y_t=Q(x,y),\quad P(x,y),\quad  Q(x,y)\in \mathbb{C}[x,y].
\end{gathered}
\end{equation}
By $\mathbb{C}[x,y]$ we denote the ring of polynomials in the variables $x$
and $y$ with coefficients in $\mathbb{C}$. The related polynomial vector
field  is defined as
\begin{equation}
\begin{gathered}
 \label{VF}
 \mathcal{X}=P(x,y)\frac{\partial}{\partial x}+Q(x,y)\frac{\partial}{\partial y}.
\end{gathered}
\end{equation}

An algebraic curve $F(x,y)=0$, $F(x,y)\in \mathbb{C}[x,y]\setminus\mathbb{C}$
is an invariant algebraic curve   of the  vector field $\mathcal{X}$ if it
satisfies the following equation $\mathcal{X} F=\lambda(x,y)F$, i.e.
\begin{equation}
\begin{gathered}
 \label{Inx_Eq}
P(x,y)F_x+Q(x,y)F_y=\lambda(x,y) F,
\end{gathered}
\end{equation}
where $\lambda(x,y)\in \mathbb{C}[x,y]$ is  a polynomial called the cofactor
of the invariant curve. It is straightforward to find that the polynomial
$\lambda(x,y)$ is of degree at most $M-1$, where $M$ is the degree of
$\mathcal{X}$: $M=\max\{\deg P, \deg Q\}$. For simplicity let us slightly
abuse notation and call the polynomial $F(x,y)$ satisfying equation
 \eqref{Inx_Eq}  an invariant algebraic curve bearing in mind that in fact the zero set of $F(x,y)$ is under consideration.
Such a polynomial $F(x,y)$ giving an invariant algebraic curve is known as a
Darboux polynomial. Note that we should take only one representative from the
family $cF(x,y)$, where $c\in\mathbb{C}$ and $c\neq0$.

Regarding the variable $y$ as dependent and the variable $x$ as independent,
we find that the function $y(x)$ satisfies the following first--order
algebraic ordinary differential equation
\begin{equation}
\begin{gathered}
 \label{ODE_y}
P(x,y)y_x-Q(x,y)=0.
\end{gathered}
\end{equation}
 It is without loss of generality to suppose that  the
polynomials $P(x,y)$ and $Q(x,y)$ are relatively prime in $\mathbb{C}[x,y]$.
Indeed,
 assuming the converse we shall only obtain a number of additional invariant
 algebraic curves given by the common zeros of  the polynomials $P(x,y)$ and $Q(x,y)$.
 Solutions of equation \eqref{ODE_y} considered as curves in the plane can be viewed as a foliation. Rewriting
 \eqref{ODE_y} as the Pfaffian equation $\omega=0$, where $\omega=Q(x,y)dx-P(x,y)dy$ is the corresponding $1$--form,
 we see that the vector field $\mathcal{X}$ is tangent to the leaves of the foliation. Thus the problem of deriving irreducible invariant algebraic
 curves of the vector field $\mathcal{X}$ is equivalent to the problem of finding algebraic leaves of the corresponding foliation.
For more details on the theory of polynomial foliations
see~\cite{Ilyashenko}.

Our aim is to describe an algebraic approach, which can be used to perform
the classification of irreducible invariant algebraic
 curves. We say that an algebraic curve is irreducible if it is irreducible in the ring $\mathbb{C}[x,y]$. Note that extending the polynomial foliations from the affine plane $\mathbb{C}^2$ to the complex projective plane $\mathbb{C}P^2$,
 one can include into consideration the behavior of foliations at infinity. This extension  can be used to obtain a number of powerful results, see
 \cite{Ilyashenko}. Similarly, the major role in the method of Puiseux series is played by Puiseux series near infinity.

Let us recall  some facts from the theory of fractional power or Puiseux
series. A Puiseux series in a neighborhood of the point $x_0$ is defined as
\begin{equation}
\begin{gathered}
 \label{Puiseux_null}
y(x)=\sum_{l=0}^{+\infty}b_l(x-x_0)^{\frac{l_0+l}{n_0}}
\end{gathered}
\end{equation}
where $l_0\in\mathbb{Z}$, $n_0\in\mathbb{N}$. In its turn a Puiseux series in
a neighborhood of the point $x=\infty$ is given by
\begin{equation}
\begin{gathered}
 \label{Puiseux_inf}
y(x)=\sum_{l=0}^{+\infty}c_lx^{\frac{l_0-l}{n_0}}
\end{gathered}
\end{equation}
where again $l_0\in\mathbb{Z}$, $n_0\in\mathbb{N}$.

It is known \cite{Walker} that a Puiseux series of the form
\eqref{Puiseux_null} that satisfy equation $F(x,y)=0$ with
$F(x,y)\in\mathbb{C}[x,y]$ is convergent in a neighborhood of the point $x_0$
(the point $x_0$ is excluded from domain of convergence if $l_0<0$).
Analogously, a Puiseux series of the form \eqref{Puiseux_inf} that satisfy
equation $F(x,y)=0$ is convergent in a neighborhood of the point $x=\infty$
(the point $x=\infty$ is excluded from domain of convergence if $l_0>0$).  If
$n_0>1$ then the convergence of the corresponding series is understood in the
sense that
 a certain branch
of the $n_0$-th root is chosen and a cut forbidding going around the branch
point is introduced.

The set of all Puiseux series of the form \eqref{Puiseux_null} or
\eqref{Puiseux_inf} forms a field, which we denote by $\mathbb{C}_{x_0}\{x\}$
or $\mathbb{C}_{\infty}\{x\}$ accordingly. In addition, we shall consider the
rings of polynomials in one variable $y$ with coefficients from the fields
$\mathbb{C}_{x_0}\{x\}$ or $\mathbb{C}_{\infty}\{x\}$. They will be referred
to as $\mathbb{C}_{x_0}\{x\}[y]$ or $\mathbb{C}_{\infty}\{x\}[y]$
respectively.

It is straightforward to prove that invariant algebraic curves of the  vector
field $\mathcal{X}$ capture Puiseux series satisfying equation \eqref{ODE_y},
\cite{Demina07}. All the Puiseux series that solve equation  \eqref{ODE_y}
can be obtained with the help
 of the power geometry and the Painlev\'{e}
 methods \cite{Bruno01, Bruno02, Conte03, Goriely}. A brief description of such a method will be given in Appendix.

Let $W(x,y)$ be a formal sum of the monomials $cx^{q_1}y^{q_2}$ with bounded
from above degrees, i.e. $q_1$, $q_2\in\mathbb{Q}$, $q_1+q_2\leq K$. Here $c$
as a constant from the field $\mathbb{C}$.  We introduce the projection
$\{W(x,y)\}_+$ giving the sum of the monomials  of $W(x,y)$ with integer
non--negative powers: $q_1$, $q_2\in\mathbb{N}\cup\{0\}$. In other words,
$\{W(x,y)\}_+$ yields the polynomial part of $W(x,y)$. It is straightforward
to show that the projection is a linear operator.

The following theorem was proved in article \cite{Demina07}.

\begin{theorem}\label{T:Darboux_pols}
Let $F(x,y)\in \mathbb{C}[x,y]\setminus\mathbb{C}$ with $F_y\not\equiv0$ be
an irreducible  invariant algebraic curve of the vector field $\mathcal{X}$
and related dynamical system \eqref{DS}. Then $F(x,y)$ takes the form
\begin{equation}
\begin{gathered}
 \label{General_F}
F(x,y)=\left\{\mu(x)\prod_{j=1}^{N}\left\{y-y_j(x)\right\}\right\}_{+},
\end{gathered}
\end{equation}
where $\mu(x)\in\mathbb{C}[x]$, and  $y_1(x)$, $\ldots$, $y_N(x)$ are
pairwise distinct Puiseux series in a neighborhood of the point $x=\infty$
that satisfy equation \eqref{ODE_y}. Moreover, the degree of $F(x,y)$ with
respect to $y$ does not exceed the number of distinct Puiseux series of the
from \eqref{Puiseux_inf} satisfying equation \eqref{ODE_y} whenever the
latter is finite.
\end{theorem}

The first result of the present article concerns the cofactor of an
irreducible invariant algebraic curve. Let us prove the following theorem.
\begin{theorem}\label{T_cof}
Let $F(x,y)$ given by  \eqref{General_F} be an irreducible  invariant
algebraic curve of the  vector field $\mathcal{X}$ and related dynamical
system \eqref{DS}. Then the cofactor $\lambda(x,y)$ of $F(x,y)$ takes the
form
\begin{equation}
\begin{gathered}
 \label{General_lambda}
\lambda(x,y)=\left\{P(x,y)\sum_{m=0}^{\infty}\sum_{l=1}^{L}\frac{\nu_lx_l^m}{x^{m+1}}+
\sum_{m=0}^{\infty}\sum_{j=1}^{N}\frac{\{Q(x,y)-P(x,y)y_{j,\,x}\}y_j^m}{y^{m+1}}\right\}_{+},
\end{gathered}
\end{equation}
where $x_1$, $\ldots$, $x_L$ are pairwise distinct zeros of the polynomial
$\mu(x)$ with multiplicities $\nu_1\in\mathbb{N}$, $\ldots$,
$\nu_L\in\mathbb{N}$ and $L\in\mathbb{N}\cup\{0\}$. If $\mu(x)=\mu_0$ where
$\mu_0\in\mathbb{C}$, then we suppose that $L=0$ and the first series is
absent in \eqref{General_lambda}.
\end{theorem}

\begin{proof}
We begin the proof  representing the irreducible invariant algebraic curve
$F(x,y)$ in the ring $\mathbb{C}_{\infty}\{x\}[y]$  as follows
\begin{equation}
\begin{gathered}
 \label{General_F_new}
F(x,y)=\mu(x)\prod_{j=1}^{N}\left\{y-y_j(x)\right\}.
\end{gathered}
\end{equation}
The polynomial $\mu(x)$ takes the form
\begin{equation}
\begin{gathered}
 \label{General_mu}
\mu(x)=\prod_{l=1}^{L}(x-x_l)^{\nu_l},\quad \nu_l\in\mathbb{N},\quad L\in\mathbb{N}\cup\{0\}.
\end{gathered}
\end{equation}
If  $\mu(x)=\mu_0$ where $\mu_0\in\mathbb{C}$, then we set $L=0$. From
\eqref{General_F_new} it follows that the  correlations
\begin{equation}
\begin{gathered}
 \label{General_F_der1}
\frac{F_x}{F}=\frac{\mu_x}{\mu}-\sum_{j=1}^{N}\frac{y_{j,\,x}}{y-y_{j}};\quad \frac{F_y}{F}=\sum_{j=1}^{N}\frac{1}{y-y_{j}}
\end{gathered}
\end{equation}
are valid. Further, we find the Laurent series of the functions $F_x/F$ and
$F_y/F$ near the point $y=\infty$. Substituting the resulting expressions
into \eqref{Inx_Eq}, we get
\begin{equation}
\begin{gathered}
 \label{General_lambda1}
\lambda(x,y)=P(x,y)\sum_{m=0}^{\infty}\sum_{l=1}^{L}\frac{\nu_lx_l^m}{x^{m+1}}+
\sum_{m=0}^{\infty}\sum_{j=1}^{N}\frac{\{Q(x,y)-P(x,y)y_{j,\,x}\}y_j^m}{y^{m+1}}
\end{gathered}
\end{equation}
where the function
\begin{equation}
\begin{gathered}
 \label{General_mu}
\frac{\mu_x}{\mu}=\sum_{l=1}^{L}\frac{\nu_l}{x-x_l}
\end{gathered}
\end{equation}
is also expanded into the Laurent series near the point $x=\infty$ Finally,
we recall that $\lambda(x,y)$ is a bivariate polynomial. This yields
\eqref{General_lambda}.

\end{proof}

Note that the roles of $x$ and $y$ in Theorems \ref{T:Darboux_pols} and
\ref{T_cof} can be changed. In addition, representations \eqref{General_F}
and \eqref{General_lambda} for an invariant algebraic curve and its cofactor
are still valid if the requirement of irreducibility is removed. But in such
a case there may arise coinciding Puiseux series in these representations.

Now let us show that the degree of the polynomial $\mu(x)$, which is the
coefficient at the highest--order term $y^N$ in the explicit expression of
$F(x,y)$, can be deduced with the help of  Puiseux series of the form
\eqref{Puiseux_null} that satisfy equation \eqref{ODE_y}.

\begin{lemma}\label{L_mu}
Let $F(x,y)\in \mathbb{C}[x,y]\setminus\mathbb{C}$ with $F_y\not\equiv0$ be
an irreducible  invariant algebraic curve of  the vector field $\mathcal{X}$
and related dynamical system \eqref{DS}. If $x_0\in\mathbb{C}$ is a zero of
the polynomial $\mu(x)$, then at least one Puiseux series of the form
\eqref{Puiseux_null} with $l_0<0$ satisfies equation \eqref{ODE_y}.

\end{lemma}

\begin{proof}
Considering the field $\mathbb{C}_{x_0}\{x\}$, which is algebraically closed
\cite{Walker}, we see that there exists uniquely determined system of
elements $y_{j,\,x_0}(x)\in \mathbb{C}_{x_0}\{x\}$ such that the following
representation is valid~\cite{Walker}
\begin{equation}
\begin{gathered}
 \label{L_mu_1}
F(x,y)=\mu(x)\prod_{j=1}^N\{y-y_{j,\,x_0}(x)\},
\end{gathered}
\end{equation}
where $N$ as before is the degree of $F(x,y)$ with respect to $y$. The set of
elements $y_{j,\,x_0}(x)\in \mathbb{C}_{x_0}\{x\}$ appearing in
representation \eqref{L_mu_1} is a subset of those satisfying
equation~\eqref{ODE_y}, see \cite{Demina07}. Let us assume that  all the
Puiseux series of the form  \eqref{Puiseux_null} that satisfy equation
\eqref{ODE_y} have $l_0\geq0$. Thus all the series in representation
\eqref{L_mu_1} also satisfy this condition. Since $x_0$ is a zero of
$\mu(x)$, we see that $F(x,y)$ given by \eqref{L_mu_1} is reducible. It is a
contradiction. Consequently,
at least one   Puiseux series near the point $x_0$ with negative exponent at the leading--order term satisfies equation \eqref{ODE_y}.\\

\end{proof}

\textit{Consequence}. If for all $ x_0\in\mathbb{C}$ the Puiseux series of
the form  \eqref{Puiseux_null} satisfying equation \eqref{ODE_y} have
non--negative exponents at the leading--order terms, then $\mu(x)=\mu_0$
where $\mu_0$ is a constant. Without loss of generality, one can set
$\mu_0=1$.

\begin{lemma}\label{L_mu2}
Suppose there exists finite number of Puiseux series of the form
\eqref{Puiseux_null} that satisfy equation \eqref{ODE_y} and have negative
exponents at the leading--order terms:
\begin{equation}
\begin{gathered}
 \label{L_mu2_1}
y_{j,\,x_0}(x)=b_0^{(j)}(x-x_0)^{-q_j}+\ldots, \quad b_0^{(j)}\neq0,\\
 q_j\in\mathbb{Q},\quad q_j>0,\quad 1\leq j\leq J\in\mathbb{N}.
\end{gathered}
\end{equation}
If $F(x,y)$ is an irreducible invariant algebraic curve of the  vector field
$\mathcal{X}$ and related dynamical system \eqref{DS}, then the multiplicity
$\nu_0\in\mathbb{N}$ of the zero of the polynomial
 $\mu(x)$ at the point $x_0$ satisfies the inequality
\begin{equation}
\begin{gathered}
 \label{L_mu2_2}
\nu_0\leq \sum_{j=1}^Jq_j.
\end{gathered}
\end{equation}

\end{lemma}

\begin{proof}
Again we consider the representation of $F(x,y)$ in the ring
$\mathbb{C}_{x_0}\{x\}[y]$, which is given by  \eqref{L_mu_1}. The Puiseux
series appearing in this representation  satisfy equation~\eqref{ODE_y}, see
\cite{Demina07}. Rewriting the polynomial $\mu(x)$ as
$\mu(x)=\mu_0(x-x_0)^{\nu_0}+\ldots$, where $\mu_0\neq0$, we see that the
factor $(x-x_0)^{\nu_0}$ should compensate the terms with negative exponents
in such a way that the resulting expression is an irreducible polynomial. The
multiplicity of $\mu(x)$ at its zero $x_0$ is the highest if all the series
given by  \eqref{L_mu2_1} are captured by representation \eqref{L_mu_1} and
other Puiseux series in \eqref{L_mu_1} have non--zero  coefficients at
$(x-x_0)^0$. Since $F(x,y)$ is irreducible, we see that each series
$y_{j,\,x_0}(x)$ given by  \eqref{L_mu2_1} can enter representation
\eqref{L_mu_1} at most once. Assuming that condition \eqref{L_mu2_2} is not
satisfied, we come to a contradiction. Indeed,  the polynomial $F(x,y)$ given
by \eqref{L_mu_1} is reducible whenever condition \eqref{L_mu2_2} is not
satisfied. This completes the proof.

\end{proof}

Now we shall study uniqueness properties of irreducible invariant algebraic
curves of the vector field $\mathcal{X}$.

\begin{theorem}\label{T:uniqueness}
Suppose a Puiseux series $y(x)$ of the form \eqref{Puiseux_null} or
\eqref{Puiseux_inf} with uniquely determined coefficients satisfies equation
\eqref{ODE_y}; then there exists at most one irreducible invariant algebraic
curve $F(x,y)$  of  the  vector field $\mathcal{X}$ and related dynamical
system \eqref{DS} such that this series is captured by $F(x,y)$, i.e.
$F(x,y(x))=0$.
\end{theorem}

\begin{proof}
The proof is by contradiction. Let us suppose that the  vector field
$\mathcal{X}$ and related dynamical system \eqref{DS} possess at least two
distinct irreducible invariant algebraic curves given by the polynomials
$F_1(x,y)$ and $F_2(x,y)$ such that $F_1(x,y(x)) =0$ and $F_2(x,y(x))=0$. We
see that these two curves intersect in an infinite number of points inside
the domain of convergence of the series $y(x)$. It follows from the
B\'{e}zout's theorem that there exists a polynomial both dividing $F_1(x,y)$
and $F_2(x,y)$. The polynomial $F_1(x,y)$ is irreducible. Thus we conclude
that $F_2(x,y)=h(x,y)F_1(x,y)$ , where $h(x,y)\in\mathbb{C}[x,y]$. But the
polynomial $F_2(x,y)$ is also irreducible. Consequently, $h\in\mathbb{C}$ and
the polynomials  $F_1(x,y)$ and $F_2(x,y)$ define the same curve. It is a
contradiction.

\end{proof}

\begin{theorem}\label{T:uniqueness2} If the number of Puiseux series of the form \eqref{Puiseux_inf} that satisfy equation  \eqref{ODE_y} is finite, then the  vector field $\mathcal{X}$
and related dynamical system \eqref{DS} possess a finite number (possibly
none) of irreducible invariant algebraic curves and, consequently, cannot
have a rational first integral. Moreover, the number of distinct irreducible
invariant algebraic curves $F(x,y)$ such that $F_y\not\equiv 0$ does not
exceed the number of distinct  Puiseux series of the form \eqref{Puiseux_inf}
that satisfy equation  \eqref{ODE_y}.
\end{theorem}

\begin{proof}
It is straightforward to show that the number of  irreducible invariant
algebraic curves that do not depend on $y$ is finite. Indeed, substituting
$F=x-x_0$, $x_0\in\mathbb{C}$ into equation \eqref{Inx_Eq}, we get
$P(x,y)=\lambda(x,y)(x-x_0)$. It follows from this relation that there may
exist only a finite number of admissible values of the parameter~$x_0$.

Further, recalling the Newton--Puiseux algorithm of finding  Puiseux series
in explicit form \cite{Walker}, we see that polynomials $F(x,y)$ representing
irreducible invariant algebraic curves in the case $F_y\not\equiv0$  cannot
involve  an arbitrary parameter. Indeed, if such a parameter existed, then it
would enter Puiseux series appearing in representation \eqref{General_F}. But
all these series satisfy equation \eqref{ODE_y} and  have no arbitrary
parameters according to the statement of the theorem. Note that  we do not
take into consideration the parameter $c$ resulting from the equivalence
$F\sim cF$. It follows from the previous observations and Theorem
\ref{T:uniqueness} that the number of distinct irreducible invariant
algebraic curves is finite. This number is the highest if all the  Puiseux
series of the form \eqref{Puiseux_inf} satisfying equation  \eqref{ODE_y} are
Laurent series that terminate at zero term. In such a case this number is
equal to the number of distinct  Puiseux series of the form
\eqref{Puiseux_inf} that satisfy equation  \eqref{ODE_y}. Finally, supposing
that  the  vector field $\mathcal{X}$ and related dynamical system \eqref{DS}
are integrable with a rational first integral, we come to a contradiction.
Indeed, in such a case  all (not some) integral curves are algebraic.

\end{proof}

\textit{Remark.} Similar theorem is valid if Puiseux series of the form
\eqref{Puiseux_null} that satisfy equation  \eqref{ODE_y} are considered.

Concluding this section let us present the resulting description of the
method of  Puiseux series.

At \textit{the first step} one should find all the Puiseux series (near
finite points and infinity) that satisfy equation  \eqref{ODE_y}.

At \textit{the second step} one uses Theorem \ref{T:Darboux_pols} and Lemmas
\ref{L_mu}, \ref{L_mu2} in order to derive the structure of an irreducible
invariant algebraic curve and its cofactor, see relations \eqref{General_F}
and  \eqref{General_lambda}. Note that at this step all possible combinations
of  Puiseux series near infinity found at the first step should be considered
if one wishes to classify irreducible invariant algebraic curves. Requiring
that the non--polynomial part of expression in brackets in \eqref{General_F}
vanishes, one obtains a system of algebraic equations. This system contains
infinite number of equation. In practice    one takes several first
equations. Note that equations resulting from the fact that the
non--polynomial part of expression in brackets in \eqref{General_lambda}
should vanish can be also considered.

At \textit{the third step} one solves the algebraic system and makes the
verification substituting the resulting  irreducible invariant algebraic
curve and its cofactor into equation \eqref{Inx_Eq}.

\textit{Remark 1.} In the above, we have considered irreducible invariant
algebraic curves $F(x,y)$ such that $F_y(x,y)\not\equiv 0$. Irreducible
invariant algebraic curves of the form $F=F(x)$ can be found substituting
expression  $F=x-x_0$ with $x_0\in\mathbb{C}$ into equation  \eqref{Inx_Eq}
and setting to zero the coefficients at different powers of $y$.

\textit{Remark 2.} If the polynomial vector field under consideration
possesses a rational first integral, then there exist infinite number of
irreducible invariant algebraic curves. In such a situation one needs to
construct  the irreducible invariant algebraic curves $F_1(x,y)$, $\ldots$,
$F_K(x,y)$  such that their cofactors $\lambda_1(x,y)$, $\ldots$,
$\lambda_K(x,y)$ satisfy the following condition \cite{Singer}
\begin{equation}
\begin{gathered}
 \label{Rat_FI}
d_1\lambda_1+\ldots+d_K\lambda_K=0,\quad d_1,\ldots,d_K\in\mathbb{Z}\setminus\{0\}, \quad K\in\mathbb{N}.
\end{gathered}
\end{equation}
All other irreducible invariant algebraic curves are expressible via
$F_1(x,y)$, $\ldots$, $F_K(x,y)$. Recall that a rational function
$R(x,y)\in\mathbb{C}(x,y)\setminus\mathbb{C}$ is a rational first integral of
the vector field $\mathcal{X}$ if $\mathcal{X}R=0$. For another method of
finding rational first integrals, which uses Taylor and Puiseux series, see
the work by A. Ferragut and H. Giacomini \cite{Ferragut01}.

An example of the method of Puiseux series in practice  will be given in
Section~\ref{S:Lienard_2_4}.

\section{Invariant algebraic curves for Li\'{e}nard dynamical systems}\label{S:Lienard}

The polynomial dynamical systems of the form
\begin{equation}
\begin{gathered}
 \label{Lienard_gen}
 x_t=y,\quad  y_t=-f(x)y-g(x)
\end{gathered}
\end{equation}
are known as Li\'{e}nard dynamical systems. In this article we suppose that
$f(x)$ and $g(x)$ are polynomials
\begin{equation}
\begin{gathered}
 \label{Lienard_fg}
f(x)=f_0x^m+\ldots+f_m, \quad g(x)=g_0x^{n}+\ldots+g_{n},\quad f_0 g_0\neq0
\end{gathered}
\end{equation}
with coefficients in the field $\mathbb{C}$. In addition we set $m<n<2m+1$.
The case $n=m+1$ was considered in \cite{Demina06, Demina11}.

Let us prove the following theorem.
\begin{theorem}\label{T1:Lienard_gen}
Let $F(x,y)\in \mathbb{C}[x,y]\setminus\mathbb{C}$ be an irreducible
invariant algebraic curve of  Li\'{e}nard dynamical system
\eqref{Lienard_gen} with $\deg f<\deg g<2\deg f+1$. Then  $F(x,y)$ and its
cofactor take the form
\begin{equation}
\begin{gathered}
 \label{Lienard1_F}
F(x,y)=\left\{\prod_{j=1}^{N-k}\left\{y-y_j(x)\right\}\left\{y-y_N(x)\right\}^{k}\right\}_{+},
\end{gathered}
\end{equation}
\begin{equation}
\begin{gathered}
 \label{Lienard1_cof}
\lambda(x,y)=-Nf-(N-k)q_x-kp_x,
\end{gathered}
\end{equation}
where $k=0$ or $k=1$, $N\in\mathbb{N}$, and $y_1(x)$, $\ldots$, $y_{N}(x)$
are the series
\begin{equation}
\begin{gathered}
 \label{Lienard1_F_series}
(I):\,y_j(x)=q(x)+\sum_{l=0}^{\infty}c_{m+1+l}^{(j)}x^{-l},\quad j=1,\ldots, N-k;\\
(II):\,y_N(x)=p(x)+\sum_{l=0}^{\infty}c_{n-m+l}^{(N)}x^{-l}.\hfill
\end{gathered}
\end{equation}
The coefficients of the series of type (II) and of the polynomials
\begin{equation}
\begin{gathered}
 \label{Lienard1_F_series_q}
q(x)=-\frac{f_0}{m+1}x^{m+1}+\sum_{l=1}^{m} q_{m+1-l} x^l\in\mathbb{C}[x],\\
 p(x)=-\frac{g_0}{f_0}x^{n-m}+\sum_{l=1}^{n-m-1} p_{n-m-l} x^l\in\mathbb{C}[x]
\end{gathered}
\end{equation}
 are uniquely determined.
The coefficients $c_{m+1}^{(j)}$, $j=1,\ldots, N-k$ are pairwise distinct.
All other coefficients
 $c_{m+1+l}^{(j)}$, $l\in\mathbb{N}$ are expressible via $c_{m+1}^{(j)}$, where $j=1,\ldots, N-k$.
 The corresponding product in  \eqref{Lienard1_F} is unit whenever $k=1$ and $N=1$.
\end{theorem}

\begin{proof}
Invariant algebraic curves of  dynamical system \eqref{Lienard_gen} satisfy
the following equation
\begin{equation}
\begin{gathered}
 \label{Lienard_gen_main_F}
yF_x-\{f(x)y+g(x)\}F_y=\lambda(x,y) F.
\end{gathered}
\end{equation}
Substituting $F=F(x)$ into this equation, we verify that there are no
invariant algebraic curves that do not depend on $y$.

Let $F(x,y)$ with $F_y\not\equiv0$  be an irreducible invariant algebraic
curve of dynamical system \eqref{Lienard_gen}. Equation \eqref{ODE_y}  now
takes the form
\begin{equation}
\begin{gathered}
 \label{Lienard_y_x}
yy_x+f(x)y+g(x)=0.
\end{gathered}
\end{equation}
For all $ x_0\in\mathbb{C}$ the Puiseux series of the form
\eqref{Puiseux_null} satisfying equation \eqref{Lienard_y_x} have
non--negative exponents at the leading--order terms. Using  Lemma \ref{L_mu},
we find $\mu(x)=\mu_0$, where $\mu_0$ is a constant. Without loss of
generality, we set $\mu_0=1$.

There exist only two dominant balances that produce Puiseux series in a
neighborhood of the point $x= \infty$. The equations related to these
balances and their solutions are the following
\begin{equation}
\begin{gathered}
 \label{Lienard_balances}
(I):\quad y(y_x+f_0x^m)=0,\quad y(x)=-\frac{f_0}{m+1}x^{m+1};\\
(II):\quad x^m(f_0y+g_0x^{n-m})=0,\quad y(x)=-\frac{g_0}{f_0}x^{n-m}.\hfill
\end{gathered}
\end{equation}
In the case $(I)$ the corresponding Puiseux series have one arbitrary
coefficient at $x^0$ provided that the compatibility condition related to the
unique Fuchs  index $l_0=m+1$ is satisfied. Let us recall that the definition
of Fuchs indices depends on the numeration of the series under consideration.
Definitions of dominant balances and Fuchs indices can be found in Appendix,
see also \cite{Conte03, Goriely, Bruno01, Bruno02, Demina07}. In this article
we use the numeration for all the series as given in \eqref{Puiseux_null} and
\eqref{Puiseux_inf}. If the parameters of system \eqref{Lienard_gen} are
chosen in such a way that the compatibility condition is not satisfied, then
the series of type $(I)$ does not exist. All the coefficients of the  series
of type $(II)$ are uniquely determined. Note that the Puiseux series are in
fact Laurent series, they are given in \eqref{Lienard1_F_series}. Factoring
$F(x,y)$ in the ring $\mathbb{C}_{\infty}\{x\}[y]$ and taking the polynomial
part of this factorization, we obtain~\eqref{Lienard1_F}.  Since the
polynomial $F(x,y)$ in ~\eqref{Lienard1_F} is irreducible, we conclude that
the series $y_1(x)$, $\ldots$, $y_{N-k}(x)$ should be pairwise distinct, see
Theorem \ref{T:Darboux_pols}. Thus the irreducibility of $F(x,y)$  requires
the coefficients
  $c_{m+1}^{(j)}$, $j=1,\ldots, N-k$
to be pairwise distinct and $k=0$ or $k=1$.

Finally, substituting $L=0$ ($\mu(x)=1$) and series \eqref{Lienard1_F_series}
into expression \eqref{General_lambda},  we find the  cofactor as given in
\eqref{Lienard1_cof}. This completes the prove.

\end{proof}

\textit{Remark.} If the parameters of system \eqref{Lienard_gen} are chosen
in such a way that the compatibility condition for the series of type $(I)$
is not satisfied, then the unique irreducible invariant algebraic curve takes
the form $F(x,y)=y-p(x)-c_{n-m}^{(N)}$ provided that the series of type
$(II)$ terminates at zero term.

Further, our aim is to show that there exist a dynamical system equivalent to
Li\'{e}nard dynamical system \eqref{Lienard_gen} such that the algebraic
ordinary differential equation related to this new system possesses finite
number of Puiseux series near infinity with respect to one of the variables.
Let us introduce the invertible  change of  variables $x=s$, $y=z+q(s)$
$\leftrightarrow$ $s=x$, $z=y-q(x)$. The polynomial $q(x) $ is  given in
\eqref{Lienard1_F_series_q}. Substituting the series of type $(I)$ into
equation  \eqref{Lienard_y_x}, we find the following recurrence relation
\begin{equation}
\begin{gathered}
 \label{Lienard1_h_l}
q_l=\frac{(m+1)}{(m-l+1)f_0}\left\{\sum_{k=0}^{l-1}q_kf_{l-k}+\sum_{k=1}^{l-1}(m+1-k)q_kq_{l-k}+g_{n+l-(2m+1)}\right\},
\end{gathered}
\end{equation}
where $1\leq l\leq m$, $q_0=-f_0/(m+1)$, $g_{-l}=0$, $l\in\mathbb{N}$ and the
second sum equals zero in the case $l=1$. The new dynamical system takes the
form
\begin{equation}
\begin{gathered}
 \label{Lienard1_alt_DS}
s_t=z+q(s),\quad z_t=-\{q_s(s)+f(s)\}\{z+q(s)\}-g(s).
\end{gathered}
\end{equation}
There exists the one--to--one correspondence between irreducible invariant
algebraic curves $F(x,y)$ of Li\'{e}nard dynamical system
\eqref{Lienard_gen} and irreducible invariant algebraic curves $G(s,z)$ of
system~\eqref{Lienard1_alt_DS}.

\begin{theorem}\label{T2_L_alt}
Suppose $G(s,z)\in \mathbb{C}[s,z]\setminus\mathbb{C}$ is an irreducible
invariant algebraic curve of dynamical system \eqref{Lienard1_alt_DS}, where
the coefficients of the polynomial $q(s) $ are given by \eqref{Lienard1_h_l}.
Then  the degree of $G(s,z)$ with respect to $s$ is either $0$ or $m+1$.
\end{theorem}
\begin{proof}
Let $G(s,z)$  be an irreducible invariant algebraic curve    of dynamical
system \eqref{Lienard1_alt_DS}. In what follows we regard $s$ as a dependent
variable and $z$ as an  independent. Our aim is to find the representation of
$G(s,z)$ in $\mathbb{C}_{\infty}\{z\}[s]$. The equation for the function
$s(z)$ reads as
\begin{equation}
\begin{gathered}
 \label{Lienard_alt_sz}
[\{q_s(s)+f(s)\}z+\{q_s(s)+f(s)\}q(s)+g(s)]s_z+q(s)+z=0.
\end{gathered}
\end{equation}
Using the definition of the polynomials $f$, $g$, and $q$, we obtain the
following relations
\begin{equation}
\begin{gathered}
 \label{Lienard_fgq}
\deg(\{q_s(s)+f(s)\}q(s)+g(s))\leq m,\\
 \deg(q_s(s)+f(s))\leq m-1,\, \deg q(s)=m+1.
\end{gathered}
\end{equation}
It follows from these relations  that equation  \eqref{Lienard_alt_sz}
possesses only one dominant balance producing Puiseux series in a
neighborhood of the point $z=\infty$. The equation related to this balance
and its solutions take the form
\begin{equation}
\begin{gathered}
 \label{Lienard_alt_dominant}
f_0s^{m+1}-(m+1)z=0,\quad s^{(k)}(z)=b_0^{(k)}z^{1/(m+1)},\quad k=1,\ldots, m+1,
\end{gathered}
\end{equation}
where $b_0^{(1)}$, $\ldots$, $b_0^{(m+1)}$ are distinct roots of the equation
$f_0 b_0^{m+1}-(m+1)=0$. The balance \eqref{Lienard_alt_dominant} is
algebraic. We find $m+1$ distinct Puiseux series in a neighborhood of the
point $z=\infty$. The representation of $G(s,z)$ in the ring
$\mathbb{C}_{\infty}\{z\}[s]$ is the following
\begin{equation}
\begin{gathered}
 \label{Lienard_alt_rep}
G(s,z)=\nu(z)\prod_{k=1}^{m+1}\{s-b_0^{(k)}z^{1/(m+1)}-\ldots\}^{n_k},
\end{gathered}
\end{equation}
where $\nu(z)\in\mathbb{C}[z]$ and $n_k=0$ or $n_k=1$. Since $G(s,z)$ should
be a polynomial, we obtain
 from representation  \eqref{Lienard_alt_rep} that  the degree of $G(s,z)$ with respect to $s$ is either $0$
 ($n_k=0$, $k=1,\ldots,m+1$) or $m+1$ ($n_k=1$, $k=1,\ldots,m+1$).
\end{proof}

The latter theorem is very important for applications, because the bound on
the degrees of irreducible invariant algebraic curves established in this
theorem can be used to find all irreducible invariant algebraic curves  of
systems \eqref{Lienard_gen} and \eqref{Lienard1_alt_DS} in explicit form. It
follows from Theorem \ref{T:uniqueness2} that the number of distinct
irreducible  invariant algebraic curves of   dynamical system
\eqref{Lienard1_alt_DS} is finite. The same is true for Li\'{e}nard dynamical
systems \eqref{Lienard_gen}, see Theorem \ref{T:L_two_curves} below.

The   change of  variables $x=s$, $y=z+q(s)$ $\leftrightarrow$ $s=x$,
$z=y-q(x)$ allows us to make a refinement in Theorem \ref{T1:Lienard_gen}.

\begin{lemma}\label{L:Lienard_gen_add}
Suppose we are in  assumptions of Theorem \ref{T1:Lienard_gen}. If  an
irreducible  invariant algebraic curve of  Li\'{e}nard dynamical system
\eqref{Lienard_gen} with $\deg f<\deg g<2\deg f+1$ does not capture  the
Puiseux series of type ($II$), then the polynomial $F(x,y)$ is of first
degree with respect to $y$.
\end{lemma}

\begin{proof}
Any irreducible  invariant algebraic curve $F(x,y)$ given by
\eqref{Lienard1_F} with $k=0$ corresponds to an irreducible invariant
algebraic curve $G(s,z)$ of dynamical system  \eqref{Lienard1_alt_DS} such
that the polynomial $G(s,z)$ does not depend on $s$. The latter, if exists,
takes the form $G(s,z)=z-z_0$ with $z_0\in\mathbb{C}$.   Consequently,
irreducible invariant algebraic curves of dynamical system
\eqref{Lienard_gen}, if any, are of the form $F(x,y)=y-q(x)-z_0$ provided
that $k=0$.
\end{proof}

\begin{theorem}\label{T:L_two_curves}
Li\'{e}nard dynamical systems \eqref{Lienard_gen} with $\deg f<\deg g<2\deg
f+1$ and fixed coefficients of the polynomials $f(x)$ and $g(x)$  have at
most two distinct irreducible invariant algebraic curves simultaneously.
\end{theorem}
\begin{proof}
First of all we recall that Li\'{e}nard dynamical systems \eqref{Lienard_gen}
have no invariant algebraic curves that do not depend on $y$. Further,  it
follows from Theorem \ref{T:uniqueness} that there exists at most one
irreducible invariant algebraic curve with $k=1$ in representation
\eqref{Lienard1_F}.

Using Lemma \ref{L:Lienard_gen_add}, we see that if $k=0$ in representation
\eqref{Lienard1_F}, then $N=1$. In this case irreducible invariant algebraic
curves, if any, can be written as $F(x,y)=y-q(x)-z_0$. If such curves exist,
then algebraic ordinary differential equation  \eqref{Lienard_y_x} should
have a polynomial solution of the form $y(x)=q(x)+z_0$. Substituting this
function into the equation in question and using the fact that the
coefficients of the polynomials $f(x)$, $g(x)$, and $q(x)$ are uniquely
determined, we find the unique value $z_0$. The same result can be obtained
if we recall that invariant algebraic curves in the case $k=0$ and $N=1$ can
be obtained requiring that the series of type ($I$) terminates at zero term.
Analyzing the recurrence relation for the coefficients of this series, which
is similar to  \eqref{Lienard1_h_l}, we see that the first appearance of
$c_{m+1}$ in the expressions for $c_{m+1+l}$, $l\in\mathbb{N}$ is with
non--zero coefficient and unit exponent.

Let us sum up our results: if the two distinct irreducible invariant
algebraic curves exist, then the first has $k=1$ in representation
\eqref{Lienard1_F} and the second is of first degree with respect to $y$ and
takes the form $F(x,y)=y-q(x)-z_0$.

\end{proof}

\textit{Consequence.} Li\'{e}nard dynamical systems \eqref{Lienard_gen} with
$\deg f<\deg g<2\deg f+1$  are not integrable with a rational first integral.

\section{Examples}\label{S:Lienard_2_4}

As an example, we shall consider  the Li\'{e}nard dynamical system of type
$\deg f=2$, $\deg g=4$:
\begin{equation}
\begin{gathered}
 \label{Lienard1_DS24_old}
x_t=y,\quad y_t=-(\zeta x^2+\beta x+\alpha)y-(\varepsilon x^4+\xi x^3+ex^2+\sigma x+\delta),\quad  \zeta\varepsilon\neq0.
\end{gathered}
\end{equation}
The change of variables $x\mapsto X(x+x_0)$, $y\mapsto Yy$, $T\mapsto Tt$,
$XYT\neq 0$ relates system \eqref{Lienard1_DS24_old} with its simplified
version at $\zeta=3$, $\varepsilon=-3$, and $\beta=0$. Thus without loss of
generality, we obtain the dynamical system
\begin{equation}
\begin{gathered}
 \label{Lienard1_DS24}
x_t=y,\quad y_t=-(3 x^2+\alpha)y+(3 x^4-\xi x^3-ex^2-\sigma x-\delta).
\end{gathered}
\end{equation}
The complete set of irreducible invariant algebraic curves for dynamical
system \eqref{Lienard1_DS24} is rather cumbersome. We shall present results
for the dynamical systems with $\xi=-2$. We have chosen this family because
there exist  irreducible invariant algebraic curves of degrees (with respect
to $y$) higher than $2$. It seems that the classification of irreducible
invariant algebraic curves for families of Li\'{e}nard dynamical system of
type $\deg f=2$, $\deg g=4$   is presented here for the first time.

\begin{theorem}\label{T_L24_1}
There unique irreducible invariant algebraic curves of  Li\'{e}nard dynamical
system \eqref{Lienard1_DS24} with $\xi=-2$ are those given in Table
\ref{Tb:L_2_4_1}.
\end{theorem}

\begin{table}[t]%[h]
        %\center
       \begin{tabular}[pos]{|l| c| l|}
        \hline
        \textit{Invariant algebraic curves} & \textit{Cofactors} & \textit{Parameters}\\
        \hline
        $ $ & $ $ & $ $\\
        $ y-x^2+\frac13(\alpha+e)$ & $-3x^2-2 x-\alpha$ & $ \delta=\frac{\alpha(e+\alpha)}{3}$, \\
        $ $ &  $ $ & $ \sigma=\frac{2(e+\alpha)}{3}$\\
        $ y+x^3+\frac32x^2+(\alpha-\frac52) x+\frac{5\alpha}{6}-\frac{\sigma}{3}-\frac{25}{12}$ & $3x-\frac52$ & $\delta=\frac{25\alpha}{12}-\frac{5\sigma}{6}-\frac{125}{24}$, \\
        $ $ & $ $ & $ e=\frac{45}{4}-3\alpha $\\
        $ y^2+(x^3+\frac12x^2-\frac{37}{4}x+\frac{35}{8})y-x^5-\frac{3}{2}x^4$ & $-3x^2+x+\frac{17}{4} $ &
        $ \alpha=-\frac{27}4$, $ e=\frac{63}{2}$,\\
        $ +\frac{35}{2}x^3+\frac{37}{4}x^2-\frac{1365}{16}x+\frac{1225}{32}$ & $ $ & $\delta=-\frac{595}{16}$, $\sigma=-\frac{9}{2}$\\
       $ y^2+(x^3+\frac{1}{2}x^2+\frac{11}{4}x+\frac{51}{8})y-x^5-\frac{3}{2}x^4$ & $-3x^2+x-\frac{31}{4} $ &
        $ \alpha=\frac{21}4$, $ e=-\frac{9}{2}$,\\
        $ -\frac52x^3-\frac{35}{4}x^2-\frac{69}{16}x+\frac{153}{32}$ & $ $ & $\delta=\frac{93}{16}$, $\sigma=-\frac{17}{2}$\\
        $ y^2+(x^3+\frac{1}{2}x^2-\frac{95}{12}x+\frac{1505}{216})y-x^5$ & $-3x^2+x+\frac{35}{12} $ &
        $ \alpha=-\frac{65}{12}$, $ e=\frac{55}{2}$,\\
        $ -\frac{3}{2}x^4+\frac{275}{18}x^3+\frac{335}{108}x^2-\frac{31175}{432}x+\frac{508475}{7776}$ & $ $ & $\delta=-\frac{11825}{432}$, \\
        $ $ &  $ $ & $\sigma=-\frac{185}{18}$\\
        $ y^3+(2x^3+2x^2-\frac12x+\frac{35}{2})y^2+(x^6+x^5$ & $-3x^2+4x-\frac{29}{4} $ &
        $ \alpha=\frac{9}{4}$, $ e=\frac{9}{2}$,\\
        $-\frac54x^4 +\frac{39}{2}x^3+\frac{103}{16}x^2-\frac{335}{16}x+\frac{7325}{64})y
        $ & $ $ & $\delta=\frac{145}{16}$, $\sigma=-\frac{39}{2}$\\
         $-x^8-3x^7+\frac12x^6-\frac{63}{4}x^5-47x^4+\frac{519}{16}x^3$ &  $ $ & $ $\\
                            $ -\frac{2065}{32}x^2-\frac{16325}{64}x+\frac{36625}{256} $ &  $ $ & $ $\\
                            $ $ &  $ $ & $ $\\
                            \hline
                            $ $ &  $ $ & $ $\\
                             $ y^2+(x^3+\frac12x^2+\{\alpha-\frac52\}x+\frac{5}{6}\{5$ & $-3x^2+x $ &
        $ \alpha^2-\frac{4975}{486}\alpha+\frac{4225}{144}$  \\
        $-\frac1{27}\alpha\}) y -x^5-\frac{3}{2}x^4+\frac{5}{12}\{15-4\alpha\}x^3$ & $-\alpha-\frac{5}{2} $ &
        $=0$, $ e=\frac{45}{4}-3\alpha$,  \\
         $+\{\frac{35}{24}-\frac{187}{162}\alpha\}x^2+\{\frac{11327}{14580}\alpha-\frac{353}{108}\}x $ &  $ $ & $\delta=\frac{5101\alpha}{14580}+\frac{193}{54}$ \\
                              $+\frac{131545}{34992}-\frac{93553}{472392}\alpha $ &  $ $ & $\sigma=\frac{\alpha}9-\frac{15}{4} $\\
                               $ $ &  $ $ & $ $\\
                              \hline
    \end{tabular}
    \caption{Irreducible invariant algebraic curves of dynamical system \eqref{Lienard1_DS24} with  $\xi=-2$. The curves in the last row have complex--valued coefficients provided that the coefficients of the original system are from $\mathbb{R}$.} \label{Tb:L_2_4_1}
\end{table}

\begin{proof}
While proving this theorem we shall use results of Theorem
\ref{T1:Lienard_gen} and Lemma \ref{L:Lienard_gen_add}. The Puiseux series
given in  \eqref{Lienard1_F_series} are now the following
\begin{equation}
\begin{gathered}
 \label{Lienard24_PS1}
(I):\quad y(x)=-x^3-\frac{3}{2}x^2+\left(\xi-\alpha+\frac{9}{2}\right)x+c_3+\sum_{l=1}^{\infty}c_{l+3}x^{-l};\hfill\\
(II):\quad y(x)=x^2-\frac{1}{3}(\xi+2)x+\frac13\left(\xi+2-\alpha-e\right)+\sum_{l=1}^{\infty}a_{l+2}x^{-l}.
\end{gathered}
\end{equation}
The Puiseux series ($I$) has an arbitrary coefficient $c_3$ and exists
whenever $e=3(27+6\xi-4\alpha)/4$. The Puiseux series ($II$) possesses
uniquely determined coefficients. Note that we use novel designations for the
coefficients of the Puiseux series ($II$). The general structure of
irreducible invariant algebraic curves and their cofactors take the form
\begin{equation}
\begin{gathered}
 \label{Lienard24_F}
F(x,y)=\left\{\prod_{j=1}^{N-k}\left\{y+x^3+\frac{3}{2}x^2-\left(\xi-\alpha+\frac{9}{2}\right)x-c^{(j)}_3-\ldots\right\}\right.\\
\left.\times \left\{y-x^2-\ldots\right\}^{k}\right\}_{+},
\end{gathered}
\end{equation}
\begin{equation}
\begin{gathered}
 \label{Lienard24_cof}
\lambda(x,y)=-3kx^2+(3N-5k)x+\frac16(8\xi+31-6\alpha)k-\frac12(2\xi+9)N,
\end{gathered}
\end{equation}
where $k=0$ or $k=1$, $N\in\mathbb{N}$.  If $e\neq3(27+6\xi-4\alpha)/4$, then
it follows from Theorem \ref{T1:Lienard_gen} that the dynamical system under
consideration possesses only one irreducible invariant algebraic curve
\begin{equation}
\begin{gathered}
 \label{Lienard24_Invc_1}
F(x,y)=y-x^2+\frac{1}{3}(\xi+2)x-\frac13\left(\xi+2-\alpha-e\right),\\
 \lambda(x,y)=-3x^2-2x-\alpha+\frac13(\xi+2)
\end{gathered}
\end{equation}
existing whenever the series of type ($II$) terminates at zero term. This
gives the following restrictions on the parameters
\begin{equation}
\begin{gathered}
 \label{Lienard24_Invc_1con}
\delta=\frac{1}{9}(\xi+2-\alpha-e)(2-3\alpha+\xi),\,\sigma=\frac19(12\alpha+6e+3\xi\alpha-\xi^2-10\xi-16).
\end{gathered}
\end{equation}
Further, we set $e=3(27+6\xi-4\alpha)/4$. If $k=0$ in relation
\eqref{Lienard24_F}, then it follows from Lemma \ref{L:Lienard_gen_add} that
$N=1$ and the irreducible invariant algebraic curve takes the form
\begin{equation}
\begin{gathered}
 \label{Lienard24_Invc_2}
F(x,y)=y+x^3+\frac{3}{2}x^2-\left(\xi-\alpha+\frac{9}{2}\right)x-\frac13(\sigma+\xi^2-\xi\alpha)\\
 -\frac14(27+12\xi-6\alpha),\quad \lambda(x,y)=3x-\xi-\frac92
\end{gathered}
\end{equation}
provided that series ($I$) terminates at zero term. This series terminates
under the condition
\begin{equation}
\begin{gathered}
 \label{Lienard24_Invc_2con}
\delta=\frac{1}{24}(2\xi+9)(18\alpha+4\xi\alpha-4\sigma-4\xi^2-36\xi-81).
\end{gathered}
\end{equation}

Now we consider the case $k=1$ and $N>1$. We calculate several first
coefficients of the Puiseux series under consideration. Requiring that the
non--polynomial part of the
 expression in brackets in  \eqref{Lienard24_F}  vanishes, we obtain necessary conditions for invariant algebraic curves to exist.
Deriving the coefficients at $y^{N-1}x^{-l}$, $l\in\mathbb{N}$ in the
representation
 of $F(x,y)$ in the field $\mathbb{C}_{\infty}\{x\}$ yields the following
 relations
  \begin{equation}
\begin{gathered}
 \label{Lienard24_alt_rep_rel1}
a_{l+2}+\sum_{j=1}^{N-1} c^{(j)}_{l+3}=0,\quad l\in\mathbb{N}.
\end{gathered}
\end{equation}
 It is convenient to introduce the variables $C_m=\{c_3^{(1)}\}^m+\ldots+\{c_3^{(N-1)}\}^m$,
 where $m\in\mathbb{N}$.  Setting $\xi=-2$, we use the algorithms of algebraic geometry, including the method of resultants
 and Gr\"{o}bner bases, to solve this system.
  In practice we need the first fifteen
equations  given by  \eqref{Lienard24_alt_rep_rel1}  in order to find all
other irreducible invariant algebraic curves. In particular, we obtain that
$N=2$ and $N=3$. Note that there exist solutions with non--natural values of
$N$. We exclude them.
 Sufficiency we verify by direct substitution into
 equation \eqref{Inx_Eq}. The results are gathered in Table~\ref{Tb:L_2_4_1}.

 Concluding the proof let us mention that the upper bound on the
degrees of irreducible invariant algebraic curves for  Li\'{e}nard dynamical
system \eqref{Lienard1_DS24} with $\xi=-2$ is equal to $8$ while the upper
bound with respect to the variable $y$ equals $3$.

\end{proof}

%Finally, let us show that there exist Li\'{e}nard dynamical system of type $\deg f=2$, $\deg g=4$
%that possess irreducible invariant algebraic curves of the fourth
%degree with respect to $y$.

%\begin{theorem}\label{L_L24_2}
%There unique irreducible invariant algebraic curves of  Li\'{e}nard dynamical system \eqref{Lienard1_DS24} with $\xi=-2$ are those given in Table \ref{Tb:L_2_4_1}.
%\end{theorem}

%\begin{proof}

%\end{proof}

\section{Conclusion}

In this article we have derived an explicit expression for the cofactor
related to an irreducible invariant algebraic curve of a polynomial dynamical
systems in the plane. We have investigated uniqueness properties of
irreducible invariant algebraic curves for planar polynomial dynamical
systems.  We have presented the general structure of irreducible invariant
algebraic curves and their  cofactors   for Li\'{e}nard dynamical systems
with $\deg f<\deg g<2\deg f+1$. We have proved that Li\'{e}nard dynamical
systems  in question can have  at most two distinct irreducible invariant
algebraic curves simultaneously. The method of the present article can be
generalized to the case of non--autonomous dynamical systems in the plane
\cite{Demina09} and higher dimensional dynamical systems \cite{Demina12}.

\section{Acknowledgments}

The author would like to thank the reviewers for their helpful comments and
suggestions.

\section{Appendix}

Let us describe a method, which can be used to perform the classification of
Puiseux series satisfying an algebraic first--order ordinary differential
equation $E(x,y,y_x)=0$. The left--hand side of this expression can be
regarded as the sum of differential monomials given~by
\begin{equation}
\begin{gathered}\label{Diff_mon}
M[y(x),x]=Cx^ly^{j_0}\left\{\frac{d y}{dx}\right\}^{j_1},\quad C\in\mathbb{C}\setminus\{0\},\quad l,j_0,j_1\in\mathbb{N}_0.
\end{gathered}
\end{equation}
The set of all the differential monomials of the form \eqref{Diff_mon} will
be denoted as $\mathbb{M}$. In order to simplify notation the expression
$W[x, y(x)]$ will stand for a polynomial in $x$, $y(x)$, and $y_x(x)$ with
coefficients from the field $\mathbb{C}$.

Let us define the map $q:$ $\mathbb{M} \rightarrow \mathbb{R}^2$ by  the
following rules
\[
Cx^{q_1}y^{q_2} \mapsto q=(q_1,q_2), \qquad
\frac{d^ky}{dx^k} \mapsto q=(-k,1),\quad q(M_1M_2)=q(M_1)+q(M_2),
\]
where $C\in \mathbb{C}\setminus\{0\}$ is a constant, $M_1$ and $M_2$ are
differential monomials. We denote the set of all points $q\in \mathbb{R}^2$
corresponding to the differential monomials of equation  $E(x,y,y_x)=0$ as
$S(E)$. The convex hull of $S(E)$ is known as \textit{the Newton polygon} of
the equation   under consideration.

The boundary of the Newton polygon consists of vertices and edges. Selecting
all the differential monomials of the original equation that generate the
vertices and the edges of the Newton polygon, we obtain a number of balances.
The balance for a vertex is defined as the sum of those differential
monomials in $E(x,y,y_x)$ that are mapped into the vertex. The balance for an
edge is defined as the sum of differential monomials in $E(x,y,y_x)$ whose
images belong to the edge.  If solutions of the equation $E(x,y,y_x)=0$
possess an asymptotics of the form $y(x)=c_0x^r$ with $x\rightarrow 0$ or
$x\rightarrow \infty$, then there exists a balance $W[x,y(x)]$ such that the
function $y(x)=c_0x^r$ satisfies the equation $W[x,y(x)]=0$. Conversely, the
function $y(x)=c_0x^r$ solving equation $W[x,y(x)]=0$, where $W[x,y(x)]$ is a
balance, is an asymptotics at $x\rightarrow 0$ (or $x\rightarrow \infty$) for
solutions of equation \eqref{ODE_y} whenever for all the differential
monomials $M[x,y(x)]$ of the original equation not involved into  $W[x,y(x)]$
we have $\text{Re}\,\varkappa>\text{Re}\,\varkappa_0$ (or
$\text{Re}\,\varkappa<\text{Re}\,\varkappa_0$), where
$M[x,c_0x^r]=Bx^{\varkappa}$ and $M_0[x,c_0x^r]=B_0x^{\varkappa_0}$ with
$M_0[x,y(x)]$ being a differential monomial of the balance $W[x,y(x)]$.

 Thus, having found all the power solutions $y(x)=c_0x^r$ for all the balances,
 one needs to select those that give asymptotics at $x\rightarrow 0$ or $x\rightarrow \infty$.
 Using power asymptotics it is possible to derive asymptotic series possessing these asymptotics as leading--order terms~\cite{Bruno01, Bruno02}.
In this article we are mainly interested in Puiseux series near $x=\infty$
that satisfy equation \eqref{ODE_y}, therefore we shall focus at the case
$r\in\mathbb{Q}$ and  $x\rightarrow \infty$. Let us suppose that a balance
$W[y(x),x]$ of the equation $E(x,y,y_x)=0$ has a solution $y(x)=c_0x^r$,
which is an asymptotics  at $x\rightarrow \infty$ and $r\in\mathbb{Q}$. In
what follows we do not consider convergence of asymptotic Puiseux series.
Note that the differentiation in the fields $\mathbb{C}_{x_0}\{x\}$ and
$\mathbb{C}_{\infty}\{x\}$ is defined as a formal operation with most of the
properties similar to those valid for convergent Puiseux series.

In order to obtain the structure  of the corresponding series one should find
the  G\^{a}teaux derivative of the balance $W[y(x),x]$ at the solution
$y(x)=c_0x^{r}$:
\[
\begin{gathered}
\frac{\delta W}{\delta y}[c_0x^{r}]=\lim_{s\rightarrow 0 }\frac{W[c_0x^{r}+sx^{r-p},x]-W[c_0x^{r},x]}{s}=V(p)x^{\tilde{r}}.
\end{gathered}
\]
In this expression $V(p)$ is a first--degree polynomial with respect to $p$.
The coefficients of this polynomial depend on $c_0$ and on the parameters (if
any) of the original equation involved into the balance $W[y(x),x]$. The zero
$p_0$ of $V(p)$ is called \textit{the Fuchs index} (or \textit{the
resonance}) of the balance $W[y(x),x]$ and its power solution
$y(x)=c_0x^{r}$. Let lcm$(n,m)$ be the lowest common multiple of two integer
numbers $n$ and $m$.
%The Fuchs index that are positive rational numbers are crucial for further analysis.
If the Fuchs index $p_0$ is not a positive rational number, then the number
$n_0$ in expression \eqref{Puiseux_inf} is given by $n_0=r_2$ where $r_2$ is
defined as $r=r_1/r_2$ with $r_1$ and $r_2$ being coprime numbers, $r_1\in
\mathbb{Z}$ and $r_2\in \mathbb{N}$. Otherwise we obtain
$n_0=\text{lcm}(g_2,r_2)$, where $r_2$ was defined previously and $g_2$ is
given by $p_0=g_1/g_2$ with coprime natural numbers  $g_1$ and $g_2$.

Finally, it is important to verify the existence of the Puiseux series of the
form \eqref{Puiseux_inf} with $l_0=rn_0$. If the balance $W[y(x),x]$
corresponds to a vertex of the Newton polygon, then the Puiseux series always
exists and possesses an arbitrary coefficient $c_0$. In this case the Fuchs
index is equal to zero. Now let us suppose that the balance $W[y(x),x]$
corresponds to an edge of the Newton polygon. Substituting series
\eqref{Puiseux_inf} into the equation $E(x,y,y_x)=0$ one can find the
recurrence relation for its coefficients. This relation takes  the form
\[
\begin{gathered}
V\left(\frac{k}{n_0}\right)c_k=U_k(c_0,\ldots,c_{k-1}),\quad k\in\mathbb{N},
\end{gathered}
\]
where $U_k$ is a polynomial of its arguments. Note that $U_k$ can also depend
on the parameters (if any) of the original equation. The equation
$U_{n_0p_0}=0$ is called \textit{the compatibility condition}. If  the
compatibility condition is not satisfied, then the Puiseux series under
consideration does not exist. Otherwise the corresponding  Puiseux series
exists and possesses an
 arbitrary coefficient $c_{n_0p_{0}}$. Consequently, we conclude that the
 Puiseux series in question has uniquely determined coefficients provided that there are no non--negative rational Fuchs indices.

We note that if one wishes to find all the Puiseux series of the form
\eqref{Puiseux_inf} that satisfy the original equation, then it is necessary
to implement the procedure described above for all the dominant balances and
for all their power solutions $y(x)=c_0x^r$ with $r\in\mathbb{Q}$ and
$x\rightarrow\infty$.

We also observe that there may exist balances and their power solutions such
that $V(p)\equiv0$. If $V(p)$ is identically zero, then one should make the
substitution $y(x)=c_0x^{r}+w(x)$ in equation $E(x,y,y_x)=0$ and find all the
Puiseux series $w(x)=c_1x^{r_1}+\ldots$ of the latter such that $r_1<r$,
$r_1\in\mathbb{Q}$ and $x\rightarrow\infty$.

If one wishes to perform the classification of Puiseux series near finite
non--zero points satisfying equation $E(x,y,y_x)=0$, then the following
variables $w(s)=y(s+x_0)$ and $s=x-x_0$ should be introduced. Making this
substitution  yields another first--order algebraic ordinary differential
equation. The method described above can be used to find balances near the
point $s=0$ for the new equation. Some examples of this method in practice
are given, for instance, in \cite{Bruno02} and \cite{Demina01, Demina02}.

\bibliography{ref_MethodPSC}

\begin{thebibliography}{10}
\expandafter\ifx\csname url\endcsname\relax
  \def\url#1{\texttt{#1}}\fi
\expandafter\ifx\csname urlprefix\endcsname\relax\def\urlprefix{URL }\fi
\expandafter\ifx\csname href\endcsname\relax
  \def\href#1#2{#2} \def\path#1{#1}\fi

\bibitem{Singer}
M.~F. Singer, Liouvillian first integrals of differential systems, Trans. Amer.
  Math. Soc. 333 (1992) 673--688.

\bibitem{Christopher02}
C.~J. Christopher, Invariant algebraic curves and conditions for a centre,
  Proc. Roy. Soc. Edinburgh Sect. A 124 (1994) 1209--1229.

\bibitem{Christopher}
C.~Christopher, Liouvillian first integrals of second order polynomial
  differential equations, Electron. J. Differential Equations 49 (1999) 1--7.

\bibitem{Ilyashenko}
Y.~Ilyashenko, S.~Yakovenko, Lectures on Analytic Differential Equations,
  Vol.~86, Graduate Studies in Mathematics, American Mathematical Society,
  2008.

\bibitem{Zhang}
X.~Zhang, Integrability of {D}ynamical {S}ystems: {A}lgebra and {A}nalysis,
  Springer Singapore, 2017.

\bibitem{Lei01}
J.~Lei, On a classification of polynomial differential operators with respect
  to the type of first integrals, J. Differential Equations 260 (2016)
  1993--2025.

\bibitem{Demina07}
M.~V. Demina, Novel algebraic aspects of {L}iouvillian integrability for
  two--dimensional polynomial dynamical systems, Physics Letters A 382~(20)
  (2018) 1353--1360.

\bibitem{Demina06}
M.~V. Demina, Invariant algebraic curves for {L}i\'{e}nard dynamical systems
  revisited, Applied Mathematics Letters 84 (2018) 42--48.

\bibitem{Demina12}
M.~V. Demina, Classifying algebraic invariants and algebraically invariant
  solutions, Chaos, Solitons and Fractals 140 (2020) 110219.

\bibitem{Odani}
K.~Odani, The limit cycle of the van der {P}ol equation is not algebraic,
  Journal of Differential Equations 115~(1) (1995) 146--152.

\bibitem{Zoladec01}
H.~\.{Z}o{\l}\c{a}dec, Algebraic invariant curves for the {L}i\'{e}nard
  equation, Trans. Amer. Math. Soc. 350 (1998) 1681--1701.

\bibitem{Llibre05}
J.~Llibre, X.~Zhang, On the algebraic limit cycles of {L}i\'{e}nard systems,
  Nonlinearity 21 (2008) 2011--2022.

\bibitem{Llibre07}
J.~Llibre, C.~Valls, Liouvillian first integrals for {L}i\'{e}nard polynomial
  differential systems, Proc. Amer. Math. Soc. 138~(9) (2010) 3229--3239.

\bibitem{Zhang02}
X.~Yu, X.~Zhang, The hyperelliptic limit cycles of the {L}i\'{e}nard systems,
  Journal of Mathematical Analysis and Applications 376 (2011) 535--539.

\bibitem{Llibre04}
J.~Gin\'{e}, J.~Llibre, Weierstrass integrability in {L}i\'{e}nard differential
  systems, Journal of Mathematical Analysis and Applications 377 (2011)
  362--369.

\bibitem{Liu01}
C.~Liu, G.~Chen, J.~Yang, On the hyperelliptic limit cycles of {L}i\'{e}nard
  systems, Nonlinearity 25~(6) (2012) 1601--1611.

\bibitem{Llibre06}
J.~Llibre, C.~Valls, Liouvillian first integrals for generalized {L}i\'{e}nard
  polynomial differential systems, Adv. Nonlinear Stud. 13 (2013) 819--829.

\bibitem{Cheze01}
G.~Ch\`{e}ze, T.~Cluzeau, On the nonexistence of {L}iouvillian first integrals
  for generalized {L}i\'{e}nard polynomial differential systems, Journal of
  Nonlinear Mathematical Physics 20~(4) (2013) 475--479.

\bibitem{Ignatyev01}
A.~O. Ignatyev, The domain of existence of a limit cycle of {L}i\'{e}nard
  system, Lobachevskii Journal of Mathematics 38~(2) (2017) 271--279.

\bibitem{Demina11}
M.~V. Demina, C.~Valls, On the {P}oincar\'{e} problem and {L}iouvillian
  integrability of quadratic {L}i\'{e}nard differential equations, Proc. Roy.
  Soc. Edinburgh Sec. A: Mathematics 150~(6) (2020) 3231--3251.

\bibitem{Walker}
R.~Walker, Algebraic Curves, Springer--Verlag, New York, 1978.

\bibitem{Bruno01}
A.~D. Bruno, Asymptotic behaviour and expansions of solutions of an ordinary
  differential equation, Russ. Math. Surv. 59~(3) (2004) 429--481.

\bibitem{Bruno02}
A.~D. Bruno, Power Geometry in Algebraic and Differential Equations, Elsevier
  Science (North--Holland), 2000.

\bibitem{Conte03}
R.~Conte, The {P}ainlev\'{e} approach to nonlinear ordinary differential
  equations, in: R.~Conte (Ed.), The {P}ainlev\'{e} Property One Century Later,
  CRM Series in Mathematical Physics, Springer--Verlag, New York, 1999, pp.
  77--180.

\bibitem{Goriely}
A.~Goriely, Integrability and Nonintegrability of Dynamical Systems, World
  Scientific, 2001.

\bibitem{Ferragut01}
A.~Ferragut, H.~Giacomini, A new algorithm for finding rational first integrals
  of polynomial vector fields, Qual. Theory Dyn. Syst. 9 (2010) 89--99.

\bibitem{Demina09}
M.~V. Demina, Invariant surfaces and {D}arboux integrability for
  non--autonomous dynamical systems in the plane, Journal of Physics A:
  Mathematical and Theoretical 51 (2018) 505202.

\bibitem{Demina01}
M.~V. Demina, N.~A. Kudryashov, Power and non--power expansions of the
  solutions for the fourth--order analogue to the second {P}ainlev\'{e}
  equation, Chaos, Solitons and Fractals 32~(1) (2007) 124--144.

\bibitem{Demina02}
M.~V. Demina, N.~A. Kudryashov, The {Y}ablonskii--{V}orob'ev polynomials for
  the second {P}ainlev\'{e} hierarchy, Chaos, Solitons and Fractals 32~(2)
  (2007) 526--537.

\end{thebibliography}

\end{document}